\documentclass[smallextended,numbook,runningheads]{svjour3}

\smartqed  % flush right qed marks, e.g. at end of proof

\usepackage{graphicx}
\usepackage{amsmath}
\usepackage{epstopdf} 
\usepackage{mathptmx} % use Times fonts if available 

\usepackage{amsfonts,tikz,hyperref}
\newcommand{\R}{\mathbb{R}}
\newcommand{\Htilde}{\widetilde H}
\newcommand{\iprod}[1]{\langle#1\rangle}
\newcommand{\Tr}{\mathcal{T}}
\newcommand{\Null}{\mathcal{N}}
\newcommand{\Range}{\mathcal{R}}
\newcommand{\LT}{\mathcal{L}}
\newcommand{\diam}{\operatorname{diam}}
\newcommand{\dist}{\operatorname{dist}}

\newcommand{\spec}{\operatorname{spec}}
\newcommand{\usng}{u_{\operatorname{s}}}
\newcommand{\ureg}{u_{\operatorname{r}}}

\newcommand{\E}{\mathcal{E}}
\newcommand{\Ehat}{\widehat{\mathcal{E}}}
\newcommand{\vecU}{\mathbf{U}}
\newcommand{\vecF}{\mathbf{F}}
\newcommand{\matM}{\mathbf{M}}
\newcommand{\matS}{\mathbf{S}}
\newcommand{\epsmix}{\epsilon_{\text{mix}}}
\newcommand{\Dxi}{\Delta\xi}
\journalname{BIT}
\begin{document}
\title{Finite element approximation of
a time-fractional diffusion problem in
a non-convex polygonal domain\thanks{This work was supported by the 
Australian Research Council grant DP140101193.}}
\titlerunning{Fractional diffusion in a nonconvex domain}

\author{Kim Ngan Le \and William McLean \and Bishnu Lamichhane}
\institute{
Kim Ngan Le
\at School of Mathematics and Statistics, The University of New South 
Wales, Sydney 2052, AUSTRALIA\\
Tel.: +612-9385-7111\\
Fax.: +612-9385-7123\\
\email{n.le-kim@unsw.edu.au}
\and
William McLean
\at School of Mathematics and Statistics, The University of New South 
Wales, Sydney 2052, AUSTRALIA\\
Tel.: +612-9385-7111\\
Fax.: +612-9385-7123\\
\email{w.mclean@unsw.edu.au}
\and
Bishnu Lamichhane
\at School of Mathematics and Physical Sciences, University of 
Newcastle, Callaghan NSW 2308, AUSTRALIA\\
Tel.: +612-4921-5515\\
Fax.: +612-4921-6898\\
\email{blamichha@gmail.com}
}
\date{\today}
\maketitle
\begin{abstract}
An initial-boundary value problem for the time-fractional diffusion 
equation is discretized in space using continuous piecewise-linear 
finite elements on a polygonal domain with a re-entrant corner.  
Known error bounds for the case of a convex polygon break down 
because the associated Poisson equation is no longer $H^2$-regular.
In particular, the method is no longer second-order accurate if
quasi-uniform triangulations are used.  We prove that a suitable 
local mesh refinement about the re-entrant corner restores 
second-order convergence.  In this way, we generalize known results 
for the classical heat equation due to Chatzipantelidis, Lazarov,
Thom\'ee~and Wahlbin.
\end{abstract}

%%%%%%%%%%%%%%%%%%%%%%%%%%%%%%%%%%%%%%%%%%%%%%%%%%%%%%%%%%%%%%%%%%%%%
\section{Introduction}
In a standard model of subdiffusion~\cite{KlafterSokolov2011}, each
particle undergoes a continuous-time random walk with a common
waiting-time distribution that obeys a power law.  
Consequently, the mean-square displacement of a particle 
is proportional to~$t^\alpha$ with~$0<\alpha<1$, and the macroscopic 
concentration~$u(x,t)$ of the particles satisfies the time-fractional 
diffusion equation
\begin{equation}\label{eq: fpde}
\partial_t u-\partial_t^{1-\alpha}K\nabla^2u=f(x,t).
\end{equation}
Here, $\partial_t=\partial/\partial t$ and $\nabla^2$ denotes the 
spatial Laplacian. The fractional time derivative is of 
Riemann--Liouville type:
\[
\partial_t^{1-\alpha}v(x,t)=\frac{\partial}{\partial t}\int_0^t
	\omega_\alpha(t-s)v(x,s)\,ds,\qquad
\omega_\alpha(t)=\frac{t^{\alpha-1}}{\Gamma(\alpha)}
\quad\text{for~$t>0$.}
\]
If no sources or sinks are present, then the 
inhomogeneous term $f$ is identically zero.  We assume for simplicity 
that the generalized diffusivity~$K$ is a positive constant, and that 
the fractional PDE~\eqref{eq: fpde} holds for~$x$ in a polygonal 
domain~$\Omega\subseteq\R^2$ subject to homogeneous Dirichlet boundary 
conditions, with the initial condition
\begin{equation}\label{eq: ic}
u(x,0)=u_0(x)\quad\text{for $x\in\Omega$.}
\end{equation}
In the limiting case when~$\alpha\to1$, the fractional 
PDE~\eqref{eq: fpde} reduces to the classical heat equation that 
arises when the diffusing particles instead undergo Brownian motion.

Consider a spatial discretization of the preceding initial-boundary 
value problem using continuous piecewise-linear finite elements to 
obtain a semidiscrete solution~$u_h$. The behaviour of~$u_h$ is well 
understood if $\Omega$ is 
convex~\cite{JinLazarovZhou2013,McLeanThomee2010}: in this case, for 
general initial data~$u_0\in L_2(\Omega)$ and an appropriate choice 
of~$u_h(0)$,
\[
\|u_h(t)-u(t)\|\le Ct^{-\alpha}h^2\|u_0\|,\qquad0<t\le T,
\]
whereas for smoother initial data~$u_0\in H^2(\Omega)$,
\[
\|u_h(t)-u(t)\|\le Ch^2\|u_0\|_{H^2(\Omega)},\qquad0\le t\le T,
\]
where ~$\|\cdot\|=\|\cdot\|_{L_2(\Omega)}$. 
The error analysis establishing these bounds relies on the 
$H^2$-regularity property of the associated elliptic equation 
in~$\Omega$, namely, that if 
\begin{equation}\label{eq: elliptic bvp}
-K\nabla^2u=f\quad\text{in~$\Omega$,}\quad
\text{with $u=0$ on~$\partial\Omega$,}
\end{equation}
then $u\in H^2(\Omega)$ with~$\|u\|_{H^2(\Omega)}\le C\|f\|$.

In the present work, our aim is to study~$u_h$ in the case when 
$\Omega$ is not convex. Since the above $H^2$-regularity breaks down,
we can no longer expect $O(h^2)$ convergence if the finite element 
mesh is quasi-uniform. Our results generalize those of 
Chatzipantelidis, Lazarov, Thom\'ee and 
Wahlbin~\cite{ChatzipantelidisEtAl2006} for the heat equation (the 
limiting case~$\alpha=1$) to the fractional-order case ($0<\alpha<1$).
Our method of analysis relies on Laplace transformation, extending the 
approach of McLean and Thomee~\cite{McLeanThomee2010} for the 
fractional order problem on a convex domain.
%%%%%%%%%%%%%%%%%%%%%%%%%%%%%%%%%%%%%%%%%%%%%%%%%%%%%%%%%%%%%%%%%%%%%%
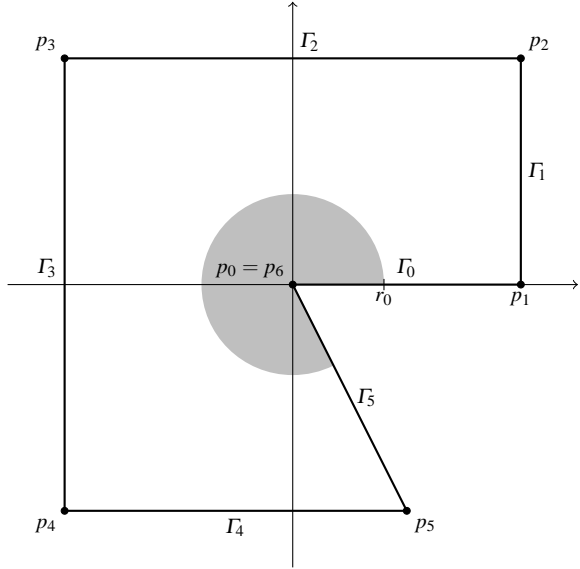
\begin{figure}
\begin{center}
\begin{tikzpicture}[scale=3.0]
\path[fill=gray!50] (0,0) -- (0.4,0) arc [radius=0.4, start angle=0, 
end 
angle=296.565] -- (0,0);
\draw[thick] (0,0) -- (1,0) -- (1,1) -- (-1,1) -- (-1,-1) 
          -- (0.5,-1) -- (0,0);
\draw[->] (-1.25,0) -- (1.25,0);
\draw[->] (0,-1.25) -- (0,1.25);
\draw (0.4,-0.025) -- (0.4,0.025);
\node[below] at (0.4,0) {$r_0$};
\newcommand*{\dotsz}{0.015}
\draw[fill] (0,0) circle [radius=\dotsz];
\node[above left] at (0,0) {$p_0=p_6$};
\draw[fill] (1,0) circle [radius=\dotsz];
\node[below] at (1,0) {$p_1$};
\draw[fill] (1,1) circle [radius=\dotsz];
\node[above right] at (1,1) {$p_2$};
\draw[fill] (-1,1) circle [radius=\dotsz];
\node[above left] at (-1,1) {$p_3$};
\draw[fill] (-1,-1) circle [radius=\dotsz];
\node[below left] at (-1,-1) {$p_4$};
\draw[fill] (0.5,-1) circle [radius=\dotsz];
\node[below right] at (0.5,-1) {$p_5$};
\node[above] at (0.5,0) {$\Gamma_0$};
\node[right] at (1,0.5) {$\Gamma_1$};
\node[above right] at (0,1) {$\Gamma_2$};
\node[above left] at (-1,0) {$\Gamma_3$};
\node[below] at (-0.25,-1) {$\Gamma_4$};
\node[right] at (0.25,-0.5) {$\Gamma_5$};
\end{tikzpicture}
\end{center}
\caption{A non-convex polygonal domain with the 
region~\eqref{eq: 0 nbhd} shaded.}
\label{fig: Omega}
\end{figure}
%%%%%%%%%%%%%%%%%%%%%%%%%%%%%%%%%%%%%%%%%%%%%%%%%%%%%%%%%%%%%%%%%%%%%%

To focus on the essential difficulty, we 
assume that $\Omega$ has only a single re-entrant corner with 
angle~$\pi/\beta$ for~$1/2<\beta<1$.  Without loss of generality, 
we assume that this corner is located at the origin and that, for 
some~$r_0>0$, the intersection of~$\Omega$ with the open 
disk~$|x|<r_0$ is described in polar coordinates by 
\begin{equation}\label{eq: 0 nbhd}
0<r<r_0\quad\text{and}\quad 0<\theta<\pi/\beta, 
\end{equation}
as illustrated in  Figure~\ref{fig: Omega}. We denote the 
vertices of~$\Omega$ by $p_0=(0,0)$, $p_1$, $p_2$, \dots, 
$p_J=p_0$, and the $j$th side by
\[
\Gamma_j=(p_j,p_{j+1})
	=\{\,(1-\sigma)p_j+\sigma p_{j+1}: 0<\sigma<1\,\}
	\quad\text{for $0\le j\le J-1$.}
\]

Section~\ref{sec: elliptic} summarizes some key facts about the
singular behaviour of the solution to the elliptic 
problem~\eqref{eq: elliptic bvp}.  In 
Section~\ref{sec: FEM}, we describe a family of shape-regular 
triangulations~$\Tr_h$ (indexed by the mesh parameter~$h$) that 
depend on a local refinement parameter~$\gamma\ge1$.  The elements 
near the origin have sizes of order~$h^\gamma$, so the $\Tr_h$ are 
quasi-uniform if $\gamma=1$ but become more highly refined with 
increasing~$\gamma$.  Our error bounds will be stated in terms of the
quantity
\begin{equation}\label{eq: epsilon}
\epsilon(h,\gamma)=\begin{cases}
        
h^{\gamma\beta}/\sqrt{\gamma^{-1}-\beta},&1\le\gamma<1/\beta,\\
        h\sqrt{\log(1+h^{-1})},&\gamma=1/\beta,\\
        h/\sqrt{\beta-\gamma^{-1}},&\gamma>1/\beta,
\end{cases}
\end{equation}
which ranges in size from $O(h^\beta)$ when~$\gamma=1$ (the 
quasiuniform case) down to $O(h)$ when~$\gamma>1/\beta$.  We briefly 
review results for the finite element approximation of the elliptic 
problem, needed for our subsequent analysis: the error in 
$H^1(\Omega)$ is of order~$\epsilon(h,\gamma)$, and the error 
in~$L_2(\Omega)$ is of order~$\epsilon(h,\gamma)^2$.

Section~\ref{sec: time-dep} gathers together some pertinant facts 
about the solution of the time-dependent problem~\eqref{eq: fpde}
and its Laplace transform. Next, in Section~\ref{sec: semidiscrete}, 
we introduce the semidiscrete finite element solution~$u_h(t)$ of the 
time-dependent problem, and see that its stability properties mimic 
those of~$u(t)$. In Section~\ref{sec: error} we study first the 
homogeneous equation (i.e., the case~$f=0$), showing that the error 
in $L_2(\Omega)$ is of order $t^{-\alpha}\epsilon(h,\gamma)^2$ when 
$u_0\in L_2(\Omega)$.  For smoother initial data, the $L_2$-error is 
of order~$\epsilon(h,\gamma)^2$ uniformly for~$0\le t\le T$.  We also 
prove that for the inhomogeneous equation ($f\ne0$) with vanishing 
initial data ($u_0=0$), the error in~$L_2(\Omega)$ is of 
order~$t^{1-\alpha}\epsilon(h,\gamma)^2$.  Thus, by choosing the 
mesh refinement parameter~$\gamma>1/\beta$ we can restore 
second-order convergence in~$L_2(\Omega)$.  Section~\ref{sec: alt bc}
outlines briefly how these results are affected by different choices 
of the boundary conditions.  We conclude in Section~\ref{sec: 
numerical} with some numerical examples that illustrate our 
theoretical error bounds. 
%%%%%%%%%%%%%%%%%%%%%%%%%%%%%%%%%%%%%%%%%%%%%%%%%%%%%%%%%%%%%%%%%%%%%%
\section{Singular behaviour in the elliptic problem}
\label{sec: elliptic}
In the weak formulation of the elliptic boundary-value 
problem~\eqref{eq: elliptic bvp} we introduce the Sobolev space
\[
V=\Htilde^1(\Omega)=H^1_0(\Omega)
\]
and seek~$u\in V$ satisfying
\[
a(u,v)=\iprod{f,v}\quad\text{for all $v\in V$,}
\]
where 
\begin{equation}\label{eq: a(u,v)}
a(u,v)=K\int_\Omega\nabla u\cdot\nabla v\,dx
\quad\text{and}\quad
\iprod{f,v}=\int_\Omega fv\,dx.
\end{equation}
Here, $f$ may belong to the dual space~$V^*=H^{-1}(\Omega)$ if 
$\iprod{f,v}$ is interpreted as the duality pairing on~$V^*\times V$.
Since $a(u,v)$ is bounded and coercive on~$\Htilde^1(\Omega)$, there 
exists a unique weak solution~$u$, and 
\begin{equation}\label{eq: Htilde u}
\|u\|_{\Htilde^1(\Omega)}\le C\|f\|_{H^{-1}(\Omega)}.
\end{equation}

To understand the difficulty created by the re-entrant corner, we 
separate variables in polar coordinates to construct the functions
\[
u_n^\pm(x)=r^{\pm n\beta}\sin(n\beta\theta)\quad
  \text{for $x=(r\cos\theta,r\sin\theta)$ and $n=1$, $2$, $3$, \dots,}
\]
satisfying 
\begin{equation}\label{eq: Laplacian in sector}
\nabla^2 u_n^\pm=0\quad
  \text{for $0<r<\infty$ and $0<\theta<\pi/\beta$,}
\end{equation}
with $u_n^\pm=0$ if $\theta=0$ or $\theta=\pi/\beta$.  Introducing a 
$C^\infty$ cutoff function~$\eta$ such that
\[
\eta(x)=1\quad\text{for $|x|\le r_0/2$}
\quad\text{and}\quad
\eta(x)=0\quad\text{for $|x|\ge r_0$,}
\]
we find that
\[
\eta u_n^+\in\Htilde^1(\Omega)
\quad\text{but}\quad
\eta u_n^-\notin\Htilde^1(\Omega)
\quad\text{for all $n\ge1$,}
\]
and that 
\[
\eta u_n^+\in H^2(\Omega)\quad\text{for all $n\ge2$,}
\quad\text{but}\quad
\eta u_1^+\notin H^2(\Omega).
\]
Now consider the function~$f=-K\nabla^2(\eta u_1^+)$.  The choice 
of~$\eta$ means that $f(x)=0$ for~$|x|\le r_0/2$, and consequently 
$f$ is $C^\infty$ on~$\overline{\Omega}$.  Nevertheless, the (unique 
weak) solution of~\eqref{eq: elliptic bvp}, namely $u=\eta u_1^+$,
fails to belong to~$H^2(\Omega)$.

Put $A=-K\nabla^2$ and
\begin{equation}\label{eq: V2 def}
V^2=H^2(\Omega)\cap\Htilde^1(\Omega)
	=\{\,v\in H^2(\Omega):\text{$v=0$ on $\partial\Omega$}\,\}.
\end{equation}

\begin{theorem}
The bounded linear operator defined by the restriction
\[
A|_{V^2}:V^2\to L_2(\Omega)
\]
is one-one and has closed range.
\end{theorem}
\begin{proof}
See Grisard~\cite[Section~2.3]{Grisvard1992}.
\qed\end{proof}

Our task now is to identify the orthogonal complement 
in~$L_2(\Omega)$ of the range
\[
\Range=\{\,f\in L_2(\Omega):
	\text{$f=Au$ for some $u\in V^2$}\,\}.
\]
To this end, we define in the usual way the Hilbert space
\[
L_2(\Omega,A)=\{\,\phi\in L_2(\Omega): A\phi\in L_2(\Omega)\,\}
\]
with the graph norm~$\|\phi\|_{L_2(\Omega,A)}^2
=\|\phi\|^2+\|A\phi\|^2$.
Let $\partial_n$ denote the outward normal derivative operator.  
It can be shown that the trace map 
$\phi\mapsto\bigl(\phi|_{\Gamma_j},\partial_n\phi|_{\Gamma_j}\bigr)$,
has unique extensions from~$C^1(\overline{\Omega})$ to bounded 
linear operators~\cite[Theorems 1.4.2~and 1.5.2]{Grisvard1992}
\[
H^2(\Omega)\to H^{3/2}(\Gamma_j)\times H^{1/2}(\Gamma_j)
\quad\text{and}\quad
L_2(\Omega,A)\to
\Htilde^{-1/2}(\Gamma_j)\times\Htilde^{-3/2}(\Gamma_j),
\]
and that the second Green identity holds in the 
form~\cite[Theorem~1.5.3]{Grisvard1992}
\[
\int_\Omega\bigl[(Au)v-u(Av)\bigr]\,dx=\sum_{j=0}^{J-1}
	K\bigl[\iprod{u,\partial_nv}_{\Gamma_j}
		-\iprod{\partial_nu,v}_{\Gamma_j}\bigr]
\]
for $u\in H^2(\Omega)$ and $v\in L_2(\Omega,A)$. Hence,
\[
\iprod{Au,\phi}=\iprod{u,A\phi}\quad
	\text{if $u\in V^2$, $\phi\in L_2(\Omega,A)$ and
	$\phi|_{\Gamma_j}=0$ for all $j$,}
\]
implying that $\Range$ is orthogonal in~$L_2(\Omega)$ to the closed 
subspace
\[
\Null=\{\,\phi\in L_2(\Omega,A):
	\text{$A\phi=0$ in $\Omega$, 
	and $\phi|_{\Gamma_j}=0$ for every~$j$}\,\}.
\]
Notice that $\Null\cap\Htilde^1(\Omega)=\{0\}$ because if~$f=0$ 
then the unique weak solution of~\eqref{eq: elliptic bvp} 
in~$\Htilde^1(\Omega)$ is $u=0$.

\begin{theorem}
The Hilbert space~$L_2(\Omega)$ is the orthogonal direct sum of 
$\Range$~and $\Null$, and $\dim\Null=1$ (assuming $\Omega$ has only a 
single re-entrant corner).
\end{theorem}
\begin{proof}
See \cite[Theorem~2.3.7]{Grisvard1992}.
\qed\end{proof}

Thus, given any $f\in L_2(\Omega)$, the unique weak 
solution~$u\in\Htilde^1(\Omega)$ of~\eqref{eq: elliptic bvp} belongs 
to~$H^2(\Omega)$ if and only if $f\perp\Null$.  For general~$f$, the 
following holds.

\begin{theorem}\label{thm: singular term}
There exists $q\in\Null$ (depending only on $\Omega$~and $\eta$)
such that if $f\in L_2(\Omega)$ then the weak solution~$u$ 
of~\eqref{eq: elliptic bvp} satisfies 
$u-\iprod{f,q}\eta u_1^+\in V^2$ with
\[
\bigl\|u-\iprod{f,q}\eta u_1^+\bigr\|_{H^2(\Omega)}\le C\|f\|.
\]
\end{theorem}
\begin{proof}
Choose any nonzero~$\phi\in\Null$. Since 
$\eta u_1^+\in L_2(\Omega, A)$ but $\eta u_1^+\notin V^2$, we have 
$\iprod{A(\eta u_1^+),\phi}\ne0$ and may therefore define
$q=c\phi\in\Null$ by letting 
$c=1/\iprod{A(\eta u_1^+),\phi}$, so that 
$\iprod{A(\eta u_1^+),q}=1$.  Define
\[
u_1=u-\iprod{f,q}\eta u_1^+\in\Htilde^1(\Omega),
\]
and observe that $u_1$ satisfies $Au_1=f_1$ where 
$f_1=f-\iprod{f,q}A(\eta u_1^+)$.
Since $\iprod{f_1,q}=0$, the we deduce that $u_1\in H^2(\Omega)$ and 
\[
\|u_1\|_{H^2(\Omega)}\le C\|f_1\|\le C\|f\|+C|\iprod{f,q}|\le C\|f\|,
\]
because $A(\eta u_1^+)\in C^\infty(\overline{\Omega})$~and
$q\in L_2(\Omega)$.
\qed\end{proof}
%%%%%%%%%%%%%%%%%%%%%%%%%%%%%%%%%%%%%%%%%%%%%%%%%%%%%%%%%%%%%%%%%%%%%
\section{Finite element approximation}\label{sec: FEM}

Consider a family~$\Tr_h$ of shape-regular triangulations of~$\Omega$,
indexed by the maximum element diameter~$h$.  
For each element~$\triangle\in\Tr_h$, let
\[
h_\triangle=\diam(\triangle)\quad\text{and}\quad
r_\triangle=\dist(0,\triangle),
\]
and suppose that for some~$\gamma\ge1$,
\begin{equation}\label{eq: h gamma}
chr_\triangle^{1-1/\gamma}\le h_\triangle
	\le Chr_\triangle^{1-1/\gamma}
	\quad\text{whenever $h^\gamma\le r_\triangle\le1$,}
\end{equation}
with
\begin{equation}\label{eq: h gamma 0}
ch^\gamma\le h_\triangle\le Ch^\gamma
	\quad\text{whenever $r_\triangle\le h^\gamma$.}
\end{equation}
Thus, if $\gamma=1$ then the mesh is globally quasiuniform, but 
for~$\gamma>1$ the element diameter decreases from order~$h$,
when~$r_\triangle\ge1$, to order~$h^\gamma$, when~$r_\triangle\le 
h^\gamma$.  Such triangulations are widely used for elliptic problems 
on domains with re-entrant corners; see for instance Apel et 
al.~\cite[Section~3]{ApelSaendigWhiteman1996}.

For each triangulation~$\Tr_h$, we let $V_h$ denote the corresponding 
space of continuous piecewise-linear functions that vanish 
on~$\partial\Omega$, so that $V_h\subseteq V=\Htilde^1(\Omega)$.  
Since the bilinear form~\eqref{eq: a(u,v)} is bounded and coercive 
on~$V$, there exists a unique finite element solution~$u_h\in V_h$ 
defined by
\begin{equation}\label{eq: uh weak}
a(u_h,v)=\iprod{f,v}\quad\text{for all $v\in V_h$.}
\end{equation}
This solution is stable in~$\Htilde^1(\Omega)$,
\begin{equation}\label{eq: Htilde uh}
\|u_h\|_{\Htilde^1(\Omega)}\le C\|f\|_{H^{-1}(\Omega)},
\end{equation}
and satisfies the quasi-optimal error bound 
\begin{equation}\label{eq: quasioptimal}
\|u_h-u\|_{\Htilde^1(\Omega)}
	\le C\min_{v\in V_h}\|v-u\|_{\Htilde^1(\Omega)}.
\end{equation}
Let $\Pi_h:C(\overline{\Omega})\to V_h$ denote the nodal 
interpolation operator, define the seminorm
\[
|v|_{m,\Omega}=\biggl(\sum_{j_1+j_2=m}\int_\Omega
	|\partial^jv(x)|^2\,dx\biggr)^{1/2},
\]
where $\partial^j=\partial_{x_1}^{j_1}\partial_{x_2}^{j_2}$,
and recall the standard interpolation error bounds~\cite{Ciarlet2002}
\begin{equation}\label{eq: Pi_h}
|v-\Pi_hv|_{m,\triangle}\le Ch_\triangle^{2-m}|v|_{2,\triangle},
\qquad m\in\{0,1\}.
\end{equation}
The next theorem reflects the influence of the singular behaviour 
of~$u$ and the local mesh refinement parameter~$\gamma$ on the 
accuracy of the approximation~$u\approx\Pi_h u$.  

\begin{theorem}\label{thm: Pi_h}
If $f\in L_2(\Omega)$ then the solution~$u\in V$ 
of the elliptic problem~\eqref{eq: elliptic bvp} satisfies
\[
\|u-\Pi_h u\|\le Ch\epsilon(\gamma,h)\|f\|
\quad\text{and}\quad
\|u-\Pi_h u\|_{\Htilde^1(\Omega)}\le C\epsilon(\gamma,h)\|f\|,
\]
where $\epsilon(h,\gamma)$ is given by~\eqref{eq: epsilon}.
\end{theorem}
\begin{proof}
We use Theorem~\ref{thm: singular term} to split~$u$ into singular 
and regular parts,
\[
u=\usng+\ureg,\qquad \usng=\iprod{f,q}\eta u_1^+,\qquad 
	\ureg\in H^2(\Omega),
\]
with~$\|\ureg\|_{H^2(\Omega)}\le C\|f\|$, leading to a corresponding 
decomposition of the interpolation error,
\[
u-\Pi_h u=(\usng-\Pi_h\usng)+(\ureg-\Pi_h\ureg).
\]
We see from~\eqref{eq: Pi_h} that
\[
\|\ureg-\Pi_h\ureg\|\le Ch^2|\ureg|_{2,\Omega}
	\le Ch^2\|f\|\le Ch\epsilon(h,\gamma)\|f\|
\]
and
\[
|\ureg-\Pi_h\ureg|_{1,\Omega}\le Ch|\ureg|_{2,\Omega}
	\le Ch\|f\|\le C\epsilon(h,\gamma)\|f\|,
\]
so it suffices to consider $\usng-\Pi_h\usng$.  Note that
$|\partial^j\usng(x)|\le C\|f\||x|^{\beta-|j|}$ for any 
multi-index~$j$,
because $u_1^+$ is homogeneous of degree~$\beta$.

We partition the triangulation into three subsets,
\begin{gather*}
\Tr_h^1=\{\,\triangle\in\Tr_h:r_\triangle<h^\gamma\,\},\qquad
\Tr_h^2=\{\,\triangle\in\Tr_h:h^\gamma\le r_\triangle<1\,\},\\
\Tr_h^3=\{\,\triangle\in\Tr_h:r_\triangle\ge1\,\},
\end{gather*}
and write
\[
|\usng-\Pi_h\usng|_{1,\Omega}^2=S_1+S_2+S_3
\quad\text{where}\quad
S_p=\sum_{\triangle\in\Tr_h^p}|\usng-\Pi_h\usng|_{1,\triangle}^2.
\]
If $r_\triangle<h^\gamma$, then
$|\usng-\Pi_h\usng|_{1,\triangle}\le|\usng|_{1,\triangle}+
|\Pi_h\usng|_{1,\triangle}$, and we estimate separately
\[
|\usng|_{1,\triangle}^2\le C\|f\|^2\int_\triangle|x|^{2(\beta-1)}\,dx
\]
and, using \eqref{eq: h gamma 0},
\[
|\Pi_h\usng|_{1,\triangle}^2
	\le Ch_\triangle^{-2}|\Pi_hu_s|_{0,\triangle}^2
	\le Ch^{-2\gamma}\|f\|^2\int_\triangle|x|^{2\beta}\,dx.
\]
Since $|x|\le r_\triangle+h_\triangle\le Ch^\gamma$ 
for~$x\in\triangle$, we see that
\begin{equation}\label{eq: S1}
S_1\le C\|f\|^2\int_{|x|\le Ch^\gamma}|x|^{2(\beta-1)}\,dx
	+Ch^{-2\gamma}\|f\|^2\int_{|x|\le Ch^\gamma}|x|^{2\beta}\,dx
	\le Ch^{2\gamma\beta}\|f\|^2.
\end{equation}
If $h^\gamma\le r_\triangle<1$, then \eqref{eq: Pi_h} gives
\[
|\usng-\Pi_h\usng|_{1,\triangle}^2\le Ch_\triangle^2\|f\|^2
	\int_\triangle|x|^{2(\beta-2)}\,dx,
\]
and our assumption~\eqref{eq: h gamma} on the mesh implies that 
for~$x\in\triangle$,
\[
h_\triangle|x|^{\beta-2}
	\le Chr_\triangle^{1-1/\gamma}|x|^{\beta-2}
	=Ch\biggl(\frac{r_\triangle}{|x|}\biggr)^{1-1/\gamma}
	|x|^{\beta-1-1/\gamma}
	\le Ch|x|^{\beta-1-1/\gamma}
\]
so
\begin{equation}\label{eq: S2}
\begin{aligned}
S_2&\le Ch^2\|f\|^2\int_{h^\gamma\le|x|\le 1+h}
		|x|^{2(\beta-1-1/\gamma)}\,dx\\
	&\le Ch^2\|f\|^2\int_{h^\gamma}^{1+h}r^{2(\beta-1/\gamma)-1}\,dr
	\le C\epsilon(h,\gamma)^2\|f\|^2.
\end{aligned}
\end{equation}
In the remaining case~$r_\triangle\ge1$, putting
$R=\sup\{\,|x|:x\in\Omega\,\}$ we have
$1\le|x|\le R$ for $x\in\triangle$, and thus
\begin{equation}\label{eq: S3}
S_3\le\sum_{\triangle\in\Tr_h^3} 
Ch_\triangle^2\|f\|^2\int_\triangle\,dx 
	\le Ch^2\|f\|^2\int_{1\le|x|\le R}\,dx\le Ch^2\|f\|^2.
\end{equation}
Together, \eqref{eq: S1}--\eqref{eq: S3} show that
$|\usng-\Pi_h\usng|_{1,\Omega}\le C\epsilon(h,\gamma)\|f\|$.

A similar argument shows
\[
\sum_{\triangle\in\Tr_h^1}|\usng-I_h\usng|_{0,\triangle}^2
	\le C\|f\|^2\int_0^{Ch^\gamma}h^{2\gamma\beta}r\,dr
	\le Ch^{2\gamma(\beta+1)}\|f\|^2
\]
and
\begin{align*}
\sum_{\triangle\in\Tr_h^2\cup\Tr_h^3}
	|\usng-I_h\usng|_{0,\triangle}^2
	&\le Ch^4\|f\|^2\int_{h^\gamma}^{R}
		r^{2(\beta-1/\gamma)-1}r^{2(1-1/\gamma)}\,dr
	\le Ch^2\epsilon(h,\gamma)^2\|f\|^2.
\end{align*}
Hence, $\|\usng-\Pi_h\usng\|\le Ch\epsilon(h,\gamma)\|f\|$ and the 
desired bounds follow.
\qed\end{proof}

\begin{theorem}\label{thm: elliptic error}
If $f\in L_2(\Omega)$ then the finite element solution~$u_h\in V_h$ 
of the elliptic problem~\eqref{eq: elliptic bvp} satisfies
\[
\|u_h-u\|\le C\epsilon(\gamma,h)^2\|f\|
\quad\text{and}\quad
\|u_h-u\|_{\Htilde^1(\Omega)}\le C\epsilon(\gamma,h)\|f\|,
\]
where $\epsilon(h,\gamma)$ is given by~\eqref{eq: epsilon}.
\end{theorem}
\begin{proof}
The bound in~$\Htilde^1(\Omega)$ follows at once 
from~\eqref{eq: quasioptimal} and Theorem~\ref{thm: Pi_h}.
The error bound in~$L_2(\Omega)$ is proved via the usual
duality argument. In fact, given any $\phi\in L_2(\Omega)$, the dual 
variational problem
\[
a(w,\psi)=\iprod{w,\phi}\quad\text{for all $w\in\Htilde^1(\Omega)$,}
\]
has a unique solution~$\psi\in\Htilde^1(\Omega)$.  Since the bilinear 
form~$a$ is symmetric, the preceding estimate for $u-\Pi_hu$ carries 
over, with $\phi$ playing the role of~$f$, to yield
$\|\psi-\Pi_h\psi\|_{\Htilde^1(\Omega)}
\le C\epsilon(h,\gamma)\|\phi\|$. Thus,
\begin{align*}
|\iprod{u_h-u,\phi}|&=|a(u_h-u,\psi)|=|a(u_h-u,\psi-\Pi_h\psi)|\\
	&\le C\|u_h-u\|_{\Htilde^1(\Omega)}
	\|\psi-\Pi_h\psi\|_{\Htilde^1(\Omega)}
	\le C\epsilon(h,\gamma)^2\|f\|\|\phi\|,
\end{align*}
implying that $\|u_h-u\|\le C\epsilon(h,\gamma)^2\|f\|$.
\qed\end{proof}
%%%%%%%%%%%%%%%%%%%%%%%%%%%%%%%%%%%%%%%%%%%%%%%%%%%%%%%%%%%%%%%%%%%%%%
\section{The time-dependent problem}\label{sec: time-dep}
We may view $A=-K\nabla^2$ as an unbounded operator 
on~$L_2(\Omega)$ with domain~$V^2$ given by~\eqref{eq: V2 def}.
Since the associated bilinear form~\eqref{eq: a(u,v)} is symmetric 
and coercive, and since the 
inclusion~$\Htilde^1(\Omega)\subseteq L_2(\Omega)$ is compact, there
exists a complete orthonormal sequence of eigenfunctions $\phi_1$, 
$\phi_2$, $\phi_3$, \dots and corresponding real eigenvalues 
$\lambda_1$, $\lambda_2$, $\lambda_3$, \dots with 
$\lambda_j\to\infty$ as~$j\to\infty$.  Thus,
\[
A\phi_n=\lambda_n\phi_n\quad\text{and}\quad
\iprod{\phi_m,\phi_n}=\delta_{mn}\quad
\text{for all $m$, $n\in\{1,2,3,\ldots\}$,}
\]
and we may assume that 
$0<\lambda_1\le\lambda_2\le\lambda_3\le\cdots$. Moreover, 
\[
(zI-A)^{-1}:L_2(\Omega)\to L_2(\Omega)
\]
is a bounded linear operator for each complex number~$z$ not in the 
spectrum~$\spec(A)=\{\,\lambda_1,\lambda_2,\lambda_3,\ldots\}$, and 
given any~$\theta_0\in(0,\pi)$ we have a resolvent estimate in 
the induced operator norm~\cite[Lemma~1]{LeGiaMcLean2011},
\begin{equation}\label{eq: resolvent}
\|(zI-A)^{-1}\|\le\frac{1+2/\lambda_1}{\sin\theta_0}\,\frac{1}{1+|z|}
	\quad\text{for $|\arg z|>\theta_0$.}
\end{equation}

Define the Laplace transform~$\hat f=\LT f$ of a 
suitable~$f:[0,\infty)\to L_2(\Omega)$ by
\[
\hat f(z)=(\LT f)(z)=\LT\{f(t)\}_{t\to z}
  =\int_0^\infty e^{-zt}f(t)\,dt,
\]
for~$\Re z$ sufficiently large. Since 
$\LT\{\partial_t^{1-\alpha}f\}_{t\to z}=z^{1-\alpha}\hat f(z)$, a 
formal calculation implies that the fractional diffusion 
equation~\eqref{eq: fpde} transforms to an elliptic problem (with 
complex coefficients) for~$\hat u(z)$,
\[
z\hat u(z)+z^{1-\alpha}A\hat u(z)=u_0+\hat f(z),
\]
and so
\begin{equation}\label{eq: u hat}
\hat u(z)=z^{\alpha-1}(z^\alpha I+A)^{-1}\bigl(u_0+\hat f(z)\bigr).
\end{equation}
The boundary condition~$u(t)=0$ on~$\partial\Omega$ tranforms to give
$\hat u(z)=0$ on~$\partial\Omega$.
Using $\LT\{t^{p\alpha}/\Gamma(1+p\alpha)\}_{t\to z}=z^{-1-p\alpha}$ 
we find that for $\lambda>0$~and $|z|>\lambda^{-1/\alpha}$,
\[
z^{\alpha-1}(z^\alpha+\lambda)^{-1}
	=z^{-1}\sum_{p=0}^\infty\bigl(-\lambda z^{-\alpha}\bigr)^p
	=\LT\biggl\{\sum_{p=0}^\infty
		\frac{(-\lambda t^\alpha)^p}{\Gamma(1+p\alpha)}
		\biggr\}_{t\to z}
	=\LT\{E_\alpha(-\lambda t^\alpha)\},
\]
where $E_\alpha$ is the Mittag--Leffler function.  
The inequalities~$0\le E_\alpha(-t)\le1$ for~$0\le t<\infty$ 
imply that the sum
\begin{equation}\label{eq: E(t)}
\E(t)v=\sum_{n=1}^\infty E_\alpha(-\lambda_nt^\alpha)
	\iprod{v,\phi_n}\phi_n
\end{equation}
defines a bounded linear operator~$\E(t):L_2(\Omega)\to L_2(\Omega)$ 
satisfying
\begin{equation}\label{eq: ||E(t)||}
\|\E(t)v\|\le\|v\|\quad\text{for $0\le t<\infty$.}
\end{equation}
Thus, for each eigenfunction~$\phi_n$,
\[
z^{\alpha-1}(z^\alpha I+A)^{-1}\phi_n
	=z^{\alpha-1}(z^\alpha+\lambda_n)^{-1}\phi_n
	=\LT\{E_\alpha(-\lambda_nt^\alpha)\phi_n\}_{t\to z},
\]
and we conclude that
\begin{equation}\label{eq: E hat}
\Ehat(z)=z^{\alpha-1}(z^\alpha I+A)^{-1}.
\end{equation}
Writing~\eqref{eq: u hat} 
as~$\hat u(z)=\Ehat(z)u_0+\Ehat(z)\hat f(z)$ then yields a Duhamel 
formula,
\begin{equation}\label{eq: u(t) mild}
u(t)=\E(t)u_0+\int_0^t\E(t-s)f(s)\,ds\quad\text{for $t>0$,}
\end{equation}
that serves to define the \emph{mild solution} of our 
initial-boundary value problem for~\eqref{eq: fpde}.
In particular, $\E(t)$ is the solution operator for the homogeneous 
problem ($f\equiv0$) with initial data~$u_0\in L_2(\Omega)$.  Also, 
the bound~\eqref{eq: ||E(t)||} immediately implies a stability 
estimate in~$L_2(\Omega)$ for a general, locally 
integrable~$f:[0,\infty)\to L_2(\Omega)$, namely
\[
\|u(t)\|\le\|u_0\|+\int_0^t\|f(s)\|\,ds\quad\text{for $t>0$.}
\]

%%%%%%%%%%%%%%%%%%%%%%%%%%%%%%%%%%%%%%%%%%%%%%%%%%%%%%%%%%%%%%%%%%%%%%
\section{The semidiscrete finite element solution}
\label{sec: semidiscrete}
Let $P_h$ denote the orthoprojector~$L_2(\Omega)\to V_h$, that is,
$P_hv\in V_h$ satisfies $\iprod{P_hv,w}=\iprod{v,w}$ for all 
$v\in L_2(\Omega)$~and $w\in V_h$. There exists a unique linear 
operator~$A_h:V_h\to V_h$ such that
\[
\iprod{A_hv,w}=a(v,w)\quad\text{for all $v$, $w\in V_h$,}
\]
and the operator equation~$A_hu_h=P_hf$ is 
equivalent to the variational equation~\eqref{eq: uh weak}
used to define the finite element solution~$u_h\in V_h$ of the 
elliptic problem~\eqref{eq: elliptic bvp}.  Denote the number of 
degrees of freedom by~$N=\dim V_h$ and equip~$V_h$ with the norm 
induced from~$L_2(\Omega)$.  The finite element space~$V_h$ has an 
orthonormal basis of eigenfunctions~$\Phi_1$, $\Phi_2$, \dots, 
$\Phi_N$ with corresponding real eigenvalues $\Lambda_1$, $\Lambda_2$, 
\dots, $\Lambda_N$. Thus,
\[
A_h\Phi_n=\Lambda_n\Phi_n\quad\text{and}\quad
\iprod{\Phi_m,\Phi_n}=\delta_{mn}\quad
\text{for $m$, $n\in\{1,2,\ldots,N\}$,}
\]
and we may assume that $0<\Lambda_1\le\Lambda_2\le\cdots\le\Lambda_N$.
Moreover, the resolvent
\[
(zI-A_h)^{-1}:V_h\to V_h
\]
exists for every 
$z\notin\spec(A_h)=\{\Lambda_1,\Lambda_2,\ldots,\Lambda_N\}$, and we 
have the an estimate corresponding to~\eqref{eq: resolvent}:
\begin{equation}\label{eq: resolvent h}
\|(zI-A_h)^{-1}\|\le\frac{1+2/\Lambda_1}{\sin\theta_0}
	\,\frac{1}{1+|z|}\quad\text{for $|\arg z|>\theta_0$.}
\end{equation}
Note that $\lambda_1\le\Lambda_1$ so this bound is uniform in~$h$.

The first Green identity yields the variational formulation
for~\eqref{eq: fpde},
\begin{equation}\label{eq: u(t) variational}
\iprod{\partial_tu,v}+a(\partial_t^{1-\alpha}u,v)=\iprod{f(t),v}
	\quad\text{for all $v\in\Htilde^1(\Omega)$ and $t>0$,}
\end{equation}
so we define the finite element solution~$u_h:[0,\infty)\to V_h$ by
\begin{equation}\label{eq: uh(t) variational}
\iprod{\partial_tu_h,v}+a(\partial_t^{1-\alpha}u_h,v)=\iprod{f(t),v}
	\quad\text{for all $v\in V_h$ and $t>0$,}
\end{equation}
with $u_h(0)=u_{0h}$, where $u_{0h}\in V_h$ is a suitable 
approximation to the initial data~$u_0$.
Thus, the vector of nodal values~$\vecU(t)$ satisfies the 
integro-differential equation in~$\R^N$,
\begin{equation}\label{eq: MOL}
\matM\partial_t\vecU+\matS\partial_t^{1-\alpha}
\vecU=\vecF(t) ,
\end{equation}
where $\matM$~and $\matS$ denote the $N\times N$ mass and stiffness 
matrices, respectively, and $\vecF(t)$ denotes the load vector.  In 
the limiting case as~$\alpha\to1$, when~\eqref{eq: fpde} becomes 
the heat equation, we see that \eqref{eq: MOL} reduces to the usual 
system of (stiff) ODEs arising in the method of lines.

The variational equation~\eqref{eq: uh(t) variational} is equivalent 
to
\[
\partial_tu_h+\partial_t^{1-\alpha}A_hu_h=P_hf(t)
	\quad\text{for $t>0$.}
\]
Taking Laplace transforms as in Section~\ref{sec: time-dep}, we find 
that
\[
\hat u_h(z)=z^{\alpha-1}(z^\alpha I+A_h)^{-1}
	\bigl(u_{0h}+P_h\hat f(z)\bigr)
\]
and thus
\[
u_h(t)=\E_h(t)u_{0h}+\int_0^t\E_h(t-s)P_hf(s)\,ds
	\quad\text{for $t>0$,}
\]
where
\[
\E_h(t)v=\sum_{n=1}^N E_\alpha(-\Lambda_n t^\alpha)
	\iprod{v,\Phi_n} \Phi_n.
\]
In the same way as~\eqref{eq: ||E(t)||} we have 
\begin{equation}\label{eq: ||Eh(t)||}
\|\E_h(t)v\|\le\|v\|\quad\text{for $t>0$~and $v\in v_h$,}
\end{equation}
implying that the finite element solution is stable in~$L_2(\Omega)$,
\begin{equation}\label{eq: uh stable}
\|u_h(t)\|\le\|u_{0h}\|+\int_0^t\|f(s)\|\,ds.
\end{equation}

For convenience, we put
\[
B(z)=(z^\alpha I+A)^{-1}
\quad\text{and}\quad
B_h(z)=(z^\alpha I+A_h)^{-1},
\]
which satisfy the following bounds.

\begin{lemma}\label{lem: B(z)}
If $|\arg z^\alpha|<\pi-\theta_0$, then 
\begin{enumerate}
\item 
$\|B(z)v\|\le C\|v\|/(1+|z|^\alpha)$ and
$\|B(z)v\|_{\Htilde^1(\Omega)}\le C\|v\|$
for $v\in L_2(\Omega)$;
\item
$\|B_h(z)v\|\le C\|v\|/(1+|z|^\alpha)$ and
$\|B_h(z)v\|_{\Htilde^1(\Omega)}\le C\|v\|$
for $v\in V_h$.
\end{enumerate}
\end{lemma}
\begin{proof}
First let~$v\in L_2(\Omega)$. The resolvent 
estimate~\eqref{eq: resolvent} immediately implies the desired bounds 
for $w(z)=B(z)v$ in~$L_2(\Omega)$.  To estimate the norm of~$w(z)$ 
in~$\Htilde^1(\Omega)$, observe that $Aw(z)=v-z^\alpha w(z)$ so 
by~\eqref{eq: Htilde u},
\[
\|w(z)\|_{\Htilde^1(\Omega)}\le C\|v-z^\alpha w(z)\|_{H^{-1}(\Omega)}
	\le C\|v-z^\alpha B(z)v\|\le C\|v\|.
\]
When~$v\in V_h$, the estimates for~$B_h(z)v$ follow in the same way 
from \eqref{eq: resolvent h}~and \eqref{eq: Htilde uh}.
\qed\end{proof}

Since $\LT\{\E(t)\phi_n\}_{t\to z}
=z^{\alpha-1}(z^\alpha+\lambda_n)^{-1}\phi_n=z^{\alpha-1}B(z)\phi_n$, 
the Laplace inversion formula implies that
\begin{align*}
\E(t)\phi_n&=\lim_{M\to\infty}\biggl(\frac{1}{2\pi i} 
  \int_{1-iM}^{1+iM} 
  e^{zt}z^{\alpha-1}(z^\alpha+\lambda_n)^{-1}\,dz\biggr)\phi_n
  =\frac{1}{2\pi i}\int_\Gamma e^{zt}z^{\alpha-1}B(z)\phi_n\,dz,
\end{align*}
for~$t>0$ and for any contour~$\Gamma$ that 
begins at~$\infty$ in the third quadrant, ends at~$\infty$ in the 
second quadrant and avoids the negative real axis.
The factor~$e^{zt}$ is exponentially small as~$\Re z\to-\infty$,
so \eqref{eq: ||E(t)||} and Lemma~\ref{lem: B(z)} ensure that
\begin{equation}\label{eq: E integral}
\E(t)v=\frac{1}{2\pi i}\int_\Gamma e^{zt}z^{\alpha-1}B(z)v\,dz
\quad\text{for $t>0$ and $v\in L_2(\Omega)$,}
\end{equation}
where the integral over~$\Gamma$ is absolutely convergent 
in~$L_2(\Omega)$.

Likewise, 
$\LT\{\E_h(t)\Phi_n\}_{t\to z}
=z^{\alpha-1}(z^\alpha+\Lambda_n)^{-1}\Phi_n=z^{\alpha-1}B_h(z)\Phi_n$
and we have a corresponding integral representation
\begin{equation}\label{eq: Eh integral}
\E_h(t)v=\frac{1}{2\pi i}\int_\Gamma e^{zt}z^{\alpha-1}
	B_h(z)v\,dz\quad\text{for $t>0$ and $v\in V_h$.}
\end{equation}
%%%%%%%%%%%%%%%%%%%%%%%%%%%%%%%%%%%%%%%%%%%%%%%%%%%%%%%%%%%%%%%%%%%%%%
\section{Error bounds}\label{sec: error}
\subsection{The homogeneous equation}

We now consider the error~$u_h(t)-u(t)$ in the case~$f\equiv0$.  
The main difficulty will be to estimate the difference
\begin{equation}\label{eq: E(t) error}
\E_h(t)P_hu_0-\E(t)u_0=\frac{1}{2\pi i}\int_\Gamma 
	e^{zt}z^{\alpha-1}G_h(z)u_0\,dz,
\end{equation}
where, by \eqref{eq: E integral}~and \eqref{eq: Eh integral},
$G_h(z)=B_h(z)P_h-B(z)$. We begin by estimating $G_h(z)v$.

\begin{lemma}\label{lem: Gh(z)}
If $v\in L_2(\Omega)$~and $|\arg z^\alpha|<\pi-\theta_0$, then
\[
\|G_h(z)v\|\le C\epsilon(h,\gamma)^2\|v\|
\quad\text{and}\quad
\|G_h(z)v\|_{\Htilde^1(\Omega)}\le 
C(1+|z|^\alpha)\epsilon(h,\gamma)\|v\|.
\]
\end{lemma}
\begin{proof}
Given $v\in L_2(\Omega)$, let $w(z)=B(z)v\in V$ so that
$Aw(z)=v-z^\alpha w(z)$, and let $w_h(z)\in V_h$ be the solution of
\[
A_hw_h(z)=P_h[v-z^\alpha w(z)].
\]
In this way, 
$P_hv=z^\alpha P_h w(z)+A_hw_h(z)=(z^\alpha I+A_h)w_h(z)
-z^\alpha[w_h(z)-P_hw(z)]$, and thus 
$B_h(z)P_hv=w_h(z)-z^\alpha B_h(z)[w_h(z)-P_hw(z)]$, implying that
\begin{equation}\label{eq: Gh(z)}
G_h(z)v=w_h(z)-w(z)-z^\alpha B_h(z)P_h[w_h(z)-w(z)].
\end{equation}
Lemma~\ref{lem: B(z)} shows that $\|w(z)\|\le C\|v\|/(1+|z|^\alpha)$ 
so $\|v-z^\alpha w(z)\|\le C\|v\|$.
By applying Theorem~\ref{thm: elliptic error}, with $w(z)$~and
$v-z^\alpha w(z)$ playing the roles of $u$~and $f$, respectively, we 
deduce that
\[
\|w_h(z)-w(z)\|\le C\epsilon(h,\gamma)^2\|v\|
\quad\text{and}\quad
\|w_h(z)-w(z)\|_{\Htilde^1(\Omega)}\le C\epsilon(h,\gamma)\|v\|.
\]
The result now follows from~\eqref{eq: Gh(z)} after another 
application of Lemma~\ref{lem: B(z)}.
\qed\end{proof}

\begin{theorem}\label{thm: non-smooth u0}
%Consider the fractional diffusion equation~\eqref{eq: fpde} 
%for~$t>0$ and $x\in\Omega$, subject to homogeneous Dirichlet 
%boundary conditions and the initial condition~\eqref{eq: ic}.
Assume that $f\equiv0$~and $u_0\in L_2(\Omega)$. Then the mild
solution~$u(t)=\E(t)u_0$ and its finite element 
approximation~$u_h(t)=\E_h(t)u_{0h}$ satisfy the error bounds
\[
\|u_h(t)-u(t)\|\le\|u_{0h}-P_hu_0\|
	+Ct^{-\alpha}\epsilon(h,\gamma)^2\|u_0\|
\]
and
\[
\|u_h(t)-u(t)\|_{\Htilde^1(\Omega)}
	\le Ct^{-\alpha}\|u_{0h}-P_hu_0\|
	+C(t^{-2\alpha}+t^{-\alpha})\epsilon(h,\gamma)\|u_0\|
\]
for $t>0$, where $\epsilon(h,\gamma)$ is given by~\eqref{eq: epsilon}.
\end{theorem}
\begin{proof}
We split the error into two terms,
\[
u_h(t)-u(t)=\E_h(t)\bigl(u_{0h}-P_hu_0\bigr)
	+\bigl[\E_h(t)P_hu_0-\E(t)u_0\bigr].
\]
It follows from~\eqref{eq: ||Eh(t)||} that
$\bigl\|\E_h(t)\bigl(u_{0h}-P_hu_0\bigr)\bigr\|
\le\|u_{0h}-P_hu_0\|$, and to estimate the second term we use
the integral representation~\eqref{eq: E(t) error} 
with~$\Gamma=\Gamma_+-\Gamma_-$, where $\Gamma_\pm$ is the contour 
$z=se^{\pm i3\pi/4}$ for~$0<s<\infty$. Applying Lemma~\ref{lem: Gh(z)} 
and making an obvious substitution, we find that 
\begin{align*}
\bigl\|\E_h(t)P_hu_0-\E(t)u_0\bigr\|&\le C\epsilon(h,\gamma)^2\|u_0\|
	\int_0^\infty e^{-st/\sqrt{2}}s^\alpha\,\frac{ds}{s}\\
	&=C\epsilon(h,\gamma)^2\|u_0\|t^{-\alpha}
	\int_0^\infty e^{-s/\sqrt{2}}s^\alpha\,\frac{ds}{s},
\end{align*}
which proves the first error bound of the theorem.

Choosing $\Gamma=\Gamma_+-\Gamma_-$ in the integral
representation~\eqref{eq: Eh integral} of~$\E_h(t)v$, and using 
Lemma~\ref{lem: B(z)}, we have for $v\in V_h$,
\[
\|\E_h(t)v\|_{\Htilde^1(\Omega)}\le C\int_0^\infty e^{-st/\sqrt{2}}
	s^\alpha\|v\|\,\frac{ds}{s}
	=Ct^{-\alpha}\|v\|\int_0^\infty e^{-s/\sqrt{2}}s^\alpha\,
	\frac{ds}{s}\le Ct^{-\alpha}\|v\|,
\]
so in particular, when~$v=u_{0h}-P_hu_0$,
\[
\bigl\|\E_h(t)\bigl(u_{0h}-P_hu_0\bigr)\bigr\|_{\Htilde^1(\Omega)}
\le Ct^{-\alpha}\|u_{0h}-P_hu_0\|.
\]
Finally, using \eqref{eq: E(t) error}~and Lemma~\ref{lem: Gh(z)} 
again,
\begin{multline*}
\bigl\|\E_h(t)P_hu_0-\E(t)u_0\bigr\|_{\Htilde^1(\Omega)}
	\le C\epsilon(h,\gamma)\|u_0\|
	\int_0^\infty e^{-st/\sqrt{2}}s^\alpha(1+s^\alpha)\,\frac{ds}{s}\\
	=C\epsilon(h,\gamma)\|u_0\|t^{-2\alpha}\int_0^\infty 
	e^{-s/\sqrt{2}}s^\alpha(t^\alpha+s^\alpha)\,\frac{ds}{s}
	\le C\epsilon(h,\gamma)(t^{-\alpha}+t^{-2\alpha})\|u_0\|,
\end{multline*}
proving the second error estimate of the theorem.
\qed\end{proof}

When $u_0$ is sufficiently regular we obtain an 
error bound that is uniform in~$t$.  The proof uses the Ritz 
projector~$R_h:\Htilde^1(\Omega)\to V_h$, defined by
\begin{equation}\label{eq: Ritz R_h}
a(R_hu,v)=a(u,v)\quad\text{for all $v\in V_h$,}
\end{equation}
and relies on the regularity estimate~\cite[Theorem~4.4]{McLean2010}
\begin{equation}\label{eq: Au_t reg}
t^\alpha\|A\partial_tu\|\le Ct^{\sigma\alpha-1}\|A^\sigma u_0\|
\quad\text{for $0<t\le T$ and $0\le\sigma\le2$.}
\end{equation}

\begin{theorem}\label{thm: u0 smooth}
Assume that $f\equiv0$ and $0<\delta\le1$.
If $A^{1+\delta}u_0\in L_2(\Omega)$, then
\[
\|u_h(t)-u(t)\|\le\|u_{0h}-R_hu_0\| 
	+C\delta^{-1}\epsilon(h,\gamma)^2\|A^{1+\delta}u_0\|
	\quad\text{for $0\le t\le T$.}
\]
\end{theorem}
\begin{proof}
To begin with, we permit $f\ne0$, and in the usual way, decompose the 
error as
\[
u_h(t)-u(t)=\vartheta(t)+\varrho(t),
\]
where $\vartheta(t)=u_h(t)-R_hu(t)\in V_h$~and 
$\varrho(t)=R_hu(t)-u(t)$. Furthermore,
since $\varrho(t)=\varrho(0)+\int_0^t\partial_t\varrho(s)\,ds$ it 
follows that
\[
\|\varrho(t)\|\le\|\varrho(0)\|+\int_0^t\|\partial_t\varrho(s)\|\,ds.
\]
By~\eqref{eq: uh(t) variational}, if $v\in V_h$ then
\[
\iprod{\partial_t\vartheta,v}+a(\partial_t^{1-\alpha}\vartheta,v)
=\iprod{f,v}-\iprod{\partial_tR_hu,v}-a(\partial_t^{1-\alpha}R_hu,v),
\]
and using the definition of the Ritz projector~\eqref{eq: Ritz R_h} 
followed by~\eqref{eq: u(t) variational}, we have
\[
a(\partial_t^{1-\alpha}R_hu,v)
=a(\partial_t^{1-\alpha}u,v)=\iprod{f,v}-\iprod{\partial_tu,v},
\]
so
\begin{equation}\label{eq: theta}
\iprod{\partial_t\vartheta,v}+a(\partial_t^{1-\alpha}\vartheta,v)
=-\iprod{\partial_t\varrho,v}.
\end{equation}
In other words, $\vartheta$ is the finite element solution of the 
fractional diffusion problem with source term~$-\partial_t\varrho(t)$. 
Thus, the stability estimate~\eqref{eq: uh stable} gives
\[
\|\vartheta(t)\|\le\|\vartheta(0)\|
	+\int_0^t\|\partial_t\varrho(s)\|\,ds.
\]
Theorem~\ref{thm: elliptic error} implies that
$\|v-R_hv\|\le C\epsilon(h,\gamma)^2\|Av\|$, so
\begin{align*}
\|u_h(t)-u(t)\|&\le\|\vartheta(0)\|+\|\varrho(0)\|+2\int_0^t
	\|\partial_t\varrho(s)\|\,ds\\
	&\le\|u_{0h}-R_hu_0\|+C\epsilon(h,\gamma)^2\biggl(\|Au_0\|
+\int_0^t\|A\partial_tu(s)\|\,ds\biggr),
\end{align*}
and, assuming now that $f\equiv0$, we use~\eqref{eq: Au_t reg} to 
bound the integral by~$C\delta^{-1}\|A^{1+\delta}u_0\|$.
\qed\end{proof}

Intermediate regularity of~$u_0$ ensures a milder growth of the 
error as~$t\to0$.

\begin{corollary}
Assume that $f\equiv0$ and $0<\delta\le1$.
If $0<\theta<1$ and $u_{0h}=P_hu_0$, then
\[
\|u_h(t)-u(t)\|\le C\delta^{-\theta}t^{-\alpha(1-\theta)}
\epsilon(h,\gamma)^2\|A^{(1+\delta)\theta}u_0\|
	\quad\text{for $0<t\le T$.}
\]
\end{corollary}
\begin{proof}
Since
$\|u_{0h}-R_hu_0\|=\|P_h(u_0-R_hu_0)\|\le\|u_0-R_hu_0\|
	\le C\epsilon(h,\gamma)^2\|Au_0\|$, we see that
\[
\|u_h(t)-u(t)\|\le Ct^{-\alpha}\epsilon(h,\gamma)^2\|u_0\|
\]
and
\[
\|u_h(t)-u(t)\|
	\le C\delta^{-1}\epsilon(h,\gamma)^2\|A^{1+\delta}u_0\|.
\]
By interpolation, 
$\|u_h(t)-u(t)\|\le C\bigl(
t^{-\alpha}\epsilon(h,\gamma)^2\bigr)^{1-\theta}
\bigl(\delta^{-1}\epsilon(h,\gamma)^2\bigr)^\theta
\|A^{(1+\delta)\theta}u_0\|$.
\qed\end{proof}

\subsection{The inhomogeneous equation}
When $u_0=0$~and $f\ne0$, we readily adapt the proof of 
Theorem~\ref{thm: u0 smooth} to show the following error bound;
see also McLean and Thom\'ee~\cite[Lemma~4.1]{McLeanThomee2010}.
Instead of~\eqref{eq: Au_t reg}, we now rely on the regularity 
result~\cite[Theorem~4.1]{McLean2010}
\begin{equation}\label{eq: Au reg}
\|A\E(t)v\|\le Ct^{-\alpha}\|v\|\quad\text{for $0<t\le T$.}
\end{equation}

\begin{theorem}\label{thm: u0=0}
If $u_0=u_{0h}=0$ then
\[
\|u_h(t)-u(t)\|\le Ct^{1-\alpha}\epsilon(h,\gamma)^2\biggl(
    \|f(0)\|+\int_0^t\|\partial_tf(s)\|\,ds\biggr)
    \quad\text{for $t>0$.}
\]
\end{theorem}
\begin{proof}
The equation~\eqref{eq: theta} for~$\vartheta$ holds for a 
general~$f$, and now $\vartheta(0)=\varrho(0)=0$, so stability of the 
finite element solution implies that
\[
\|u_h(t)-u(t)\|\le2\int_0^t\|\partial_t\varrho(s)\|\,ds
    \le C\epsilon(h,\gamma)^2\int_0^t\|A\partial_tu(s)\|\,ds.
\]
Since 
\[
Au(s)=A\int_0^s\E(s-\tau)f(\tau)\,d\tau
    =\int_0^sA\E(\tau)f(s-\tau)\,d\tau,
\]
we have
\[
A\partial_tu(s)=A\E(s)f(0)
    +\int_0^sA\E(\tau)\partial_tf(s-\tau)\,d\tau,
\]
so using \eqref{eq: Au reg}, 
\[
\|A\partial_tu(s)\|\le Cs^{-\alpha}\|f(0)\|
    +\int_0^s \tau^{-\alpha}\|\partial_tf(s-\tau)\|\,d\tau.
\]
Thus, 
\[
\int_0^t\|A\partial_tu(s)\|\,ds\le Ct^{1-\alpha}\|f(0)\|
+\int_0^t\int_0^s\tau^{-\alpha}\|\partial_tf(s-\tau)\|\,d\tau\,ds
\]
and the double integral equals
\[
\int_0^t\int_0^s(s-\tau)^{-\alpha}\|\partial_tf(\tau)\|\,d\tau\,ds
    =\int_0^t\|\partial_tf(\tau)\|(t-\tau)^{1-\alpha}\,d\tau,
\]
implying the desired estimate.
\qed\end{proof}

%%%%%%%%%%%%%%%%%%%%%%%%%%%%%%%%%%%%%%%%%%%%%%%%%%%%%%%%%%%%%%%%%%%%%%
\section{Alternative boundary conditions}\label{sec: alt bc}
\subsection{Neumann boundary conditions}
Separation of variables in polar coordinates yields the functions
\[
u_n^\pm=r^{\pm n\beta}\cos(n\beta\theta)
\quad\text{for $n=1$, $2$, $3$, \dots}
\]
satisfying \eqref{eq: Laplacian in sector} with $\partial_\theta 
u_n^\pm=0$ if $\theta=0$ or $\theta=\pi/\beta$.  In addition, 
for~$n=0$ we find $u_0^+=1$~and $u_0^-=\log r$, and can readily check 
that
\begin{equation}\label{eq: un H1}
\eta u_n^+\in H^1(\Omega)
\quad\text{but}\quad
\eta u_n^-\notin H^1(\Omega)
\quad\text{for all $n\ge0$,} 
\end{equation}
and that $\eta u^+_n\in H^2(\Omega)$ iff $n\ne1$.
If we impose a homogeneous Neumann boundary condition~$\partial_n 
u=0$ on~$\partial\Omega$, then our results are essentially unchanged, 
but the fact that $A=-K\nabla^2$ now possesses a zero eigenvalue 
complicates the analysis~\cite[Section~4]{MustaphaMcLean2011}.

\subsection{Mixed boundary conditions}
The functions
\[
u_n^\pm=r^{(n-\tfrac12)\beta}\sin(n-\tfrac12)\beta\theta
\quad\text{for $n=1$, $2$, $3$, \dots}
\]
satisfy \eqref{eq: Laplacian in sector} with
$u_n^\pm=0$ if $\theta=0$ and $\partial_\theta u_n^\pm=0$ if 
$\theta=\pi/\beta$.  Once again, \eqref{eq: un H1} holds, for 
all~$n\ge1$, however now we have
\[
\eta u_n^+\in H^2(\Omega)\quad\text{for all $n\ge3$,}
\quad\text{but}\quad
\eta u_1^+, \eta u_2^+\notin H^2(\Omega),
\]
assuming $1/2<\beta<1$.  A new feature is that $\eta u_1^+\notin 
H^2(\Omega)$ also when $1\le\beta<2$, that is, for an interior angle 
between $\pi/2$~and $\pi$, in which case $\Omega$ is in fact convex.
The proof of Theorem~3.1 must be modified by replacing $\beta$ 
with~$\beta/2$, and replacing $\epsilon(h,\gamma)$ with
\begin{equation}\label{eq: epsmix}
\epsmix(h,\gamma)=\begin{cases}
	h^{\gamma\beta/2}/\sqrt{\gamma^{-1}-\beta/2},&
		1\le\gamma<2/\beta,\\
	h\sqrt{\log(1+h^{-1})},&\gamma=2/\beta,\\
	h/\sqrt{\beta/2-\gamma^{-1}},&\gamma>2/\beta,
\end{cases}
\end{equation}
provided the interior angles at~$p_1$, $p_2$, \dots, $p_{J-1}$ are 
all less than or equal to~$\pi/2$.  We may then proceed as for 
Dirichlet boundary conditions (since all the eigenvalues of~$A$ are 
strictly positive), with $\epsilon(h,\gamma)$ replaced by 
$\epsmix(h,\gamma)$ in our error estimates.
%%%%%%%%%%%%%%%%%%%%%%%%%%%%%%%%%%%%%%%%%%%%%%%%%%%%%%%%%%%%%%%%%%%%%%
\section{Numerical experiments}\label{sec: numerical}
We consider two problems posed on a domain of the form
\[
\Omega=\{\,(r\cos\theta,r\sin\theta):
	\text{$0<r<1$ and $0<\theta<\pi/\beta$}\,\},
\]
with~$\beta=2/3$. Although $\Omega$ is not a polygon, the 
additional error in~$u_h$ due to approximation of the curved part 
of~$\partial\Omega$ is of order~$h^2$ in~$L_2(\Omega)$, and hence
our error bounds should apply unchanged. To fix the time 
scale for the solutions of the fractional diffusion 
equation~\eqref{eq: fpde}, we choose the generalized diffusivity~$K$ 
so that the smallest eigenvalue of~$A=-K\nabla^2$ equals~$1$. 
Figure~\ref{fig: meshes} shows two successive meshes out of a sequence 
satisfying our assumptions \eqref{eq: h gamma}~and \eqref{eq: h gamma 
0} 
for~$\gamma=1/\beta=3/2$; notice that these meshes are not nested.  
The mesh generation code takes a specified $h_*$~and $\gamma$ and 
produces a triangulation with maximum element diameter~$h$ equivalent 
to~$h_*$.  All source files were written in Julia~0.4~\cite{Julia}
with some calls to Gmsh~2.10.1~\cite{Gmsh}, and all computations 
performed on a desktop PC with~16GB of~RAM and an Intel Core i7-4770 
CPU.
%%%%%%%%%%%%%%%%%%%%%%%%%%%%%%%%%%%%%%%%%%%%%%%%%%%%%%%%%%%%%%%%%%%%%
\begin{figure}
\caption{Meshes with $h_*=2^{-3}$ (left) and $h_*=2^{-4}$ (right) 
from a sequence satisfying \eqref{eq: h gamma}~and 
\eqref{eq: h gamma 0} for~$\gamma=3/2$.}\label{fig: meshes} 
\begin{center}
\includegraphics[scale=0.35]{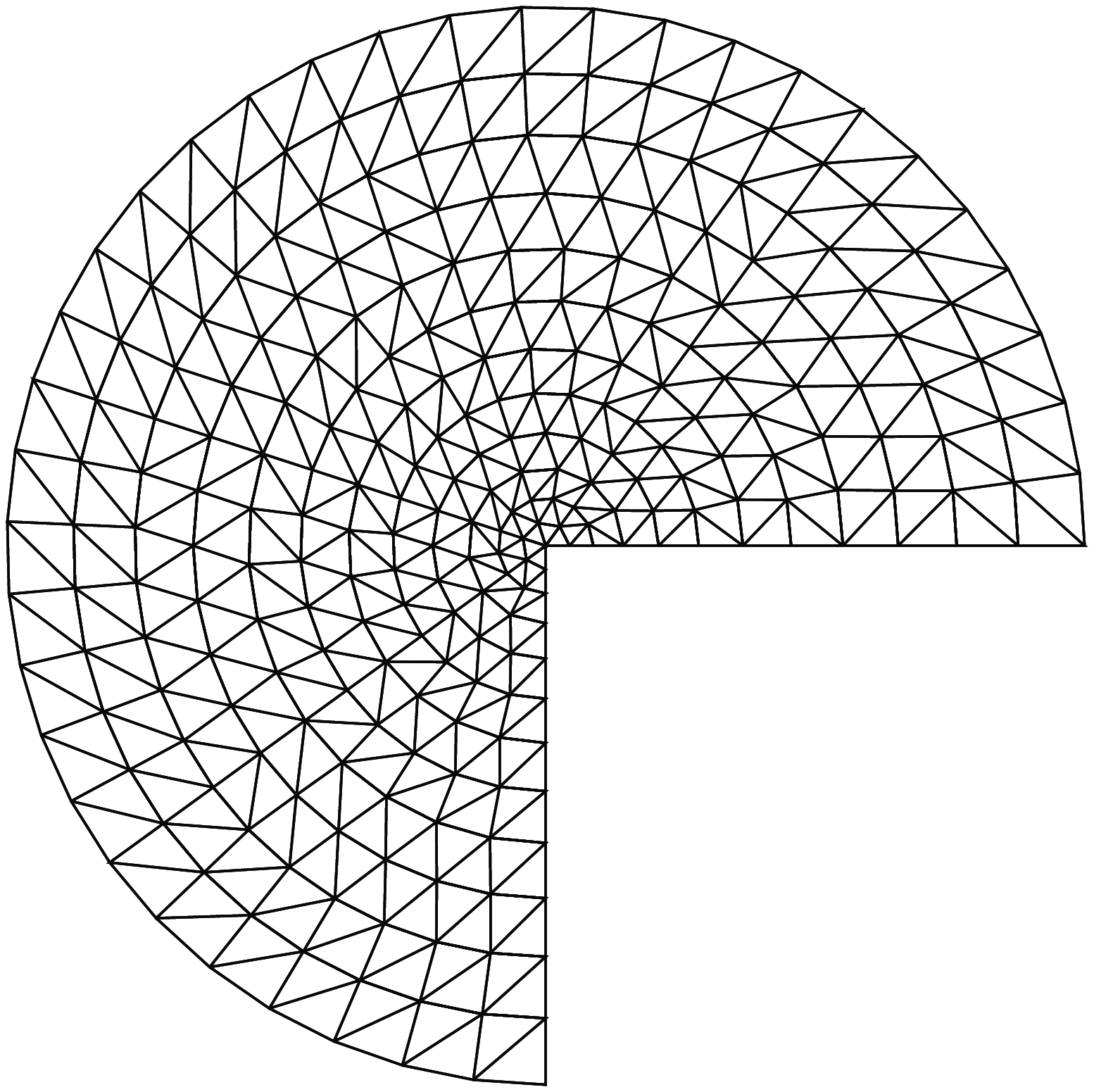}
\hfill
\includegraphics[scale=0.35]{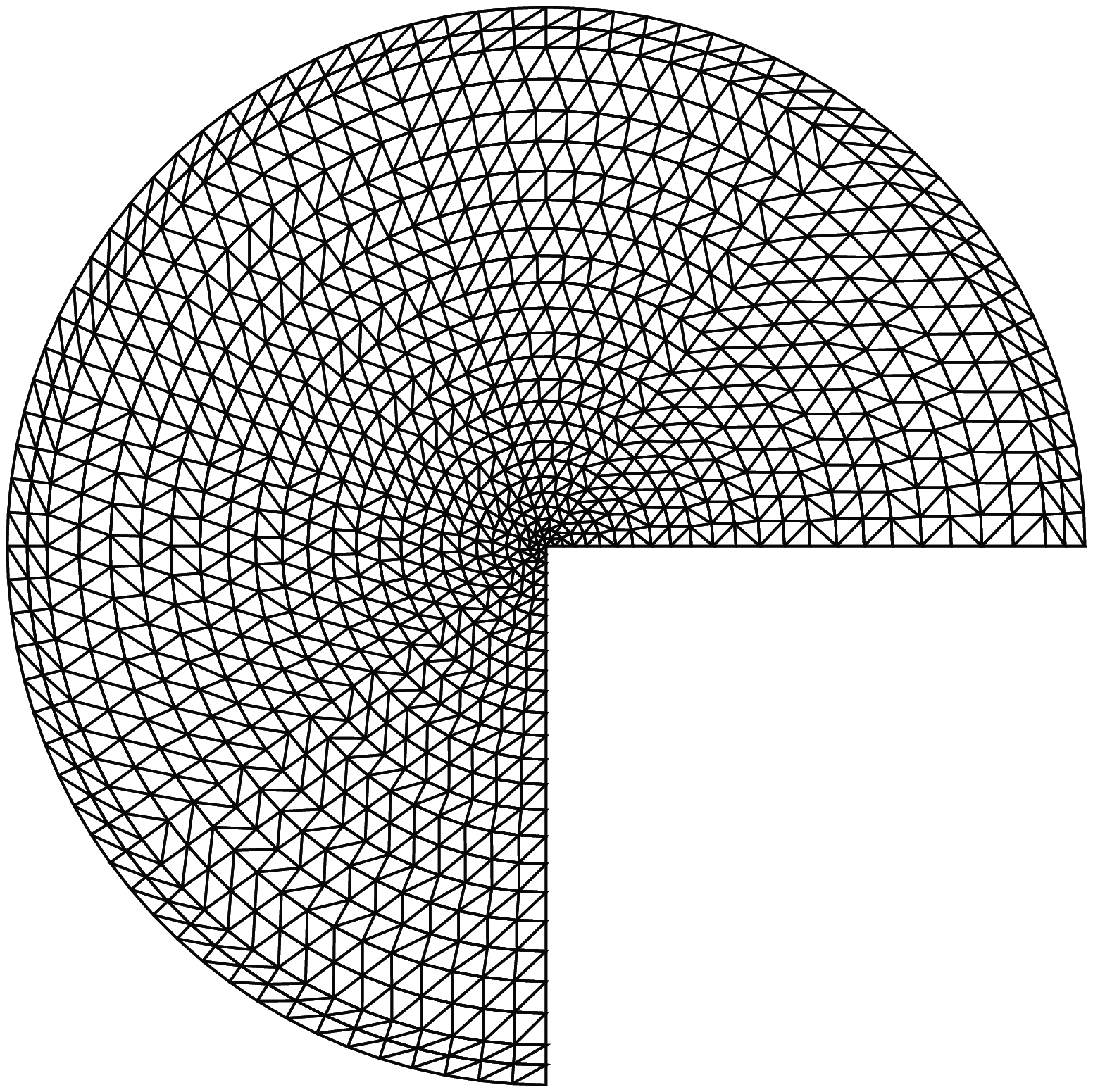}
\end{center}
\end{figure}
%%%%%%%%%%%%%%%%%%%%%%%%%%%%%%%%%%%%%%%%%%%%%%%%%%%%%%%%%%%%%%%%%%%%%

For the time integration, we use a 
technique~\cite{McLeanThomee2010,WeidemanTrefethen2007} based on a 
quadrature approximation to the Laplace inversion formula,
\[
u_h(t)=\frac{1}{2\pi i}\int_\Gamma
	e^{zt}\hat u_h(z)\,dz
	=\frac{1}{2\pi i}\int_{-\infty}^\infty 
		e^{z(\xi)t}\hat u\bigl(z(\xi)\bigr)z'(\xi)
		\,d\xi,
\]
where the contour~$\Gamma$ has the parametric representation
\[
z(\xi)=\mu\bigl(1-\sin(\delta-i\xi)\bigr)
	\quad\text{for $-\infty<\xi<\infty$,}
\]
with $\delta=1.1721\,0423$ and $\mu=4.4920\,7528\,M/t$ for 
given~$t>0$ and a chosen positive integer~$M$.
Therefore, the contour~$\Gamma$ is the left branch of an hyperbola 
with asymptotes~$y=\pm(x-\mu)\cot\delta$ for~$z=x+iy$.  Putting
\[
z_j=z(\xi_j),\quad z'_j=z'(\xi_j),\quad
\xi_j=j\,\Dxi,\quad\Dxi=\frac{1.0817\,9214}{M},
\]
we define
\[
U_{M,h}(t)=\frac{\Dxi}{2\pi i}\sum_{j=-M}^Me^{z_jt}
	\hat u_h(z_j)z'_j\approx u_h(t).
\]
To compute $\hat u_h(z_j)$ we solve the (complex) finite element 
equations
\[
z_j^\alpha\iprod{\hat u_h(z_j),\chi}+a\bigl(\hat u_h(z_j),\chi\bigr)
	=z_j^{\alpha-1}\iprod{u_{0h}+\hat f(z_j),\chi},\quad\chi\in V_h,
\]
and since we choose real $u_{0h}$~and $f$, it follows that
$\hat u_h(z_{-j})=\hat u_h(\bar z_j)=\overline{\hat u_h(z_j)}$ so the 
number of elliptic solves needed to evaluate~$U_{M,h}(t)$ is~$M+1$,
not~$2M+1$.  An error bound for the quadrature 
error~$\|U_{M,h}(t)-u_h(t)\|$ includes a decay factor~$10.1315^{-M}$, 
and we observe in practice that the overall 
error~$\|U_{M,h}(t)-u(t)\|$ is dominated by the finite element 
error~$\|u_h(t)-u(t)\|$ for modest values of~$M$.  In the computations
reported below we use $M=8$ to compute $U_{M,h}(t)\approx u_h(t)$, 
and choose $u_{0h}=P_hu_0$ for the discrete initial data.

%%%%%%%%%%%%%%%%%%%%%%%%%%%%%%%%%%%%%%%%%%%%%%%%%%%%%%%%%%%%%%%%%%%%%
\begin{figure}
\caption{Behaviour of the $L_2$-error~$\|u_h(t)-u(t)\|$ for Example~1
when~$t=1$ --- quasiuniform versus locally refined triangulations.}
\label{fig: locally-refined} 
\begin{center}
\includegraphics[scale=0.4]{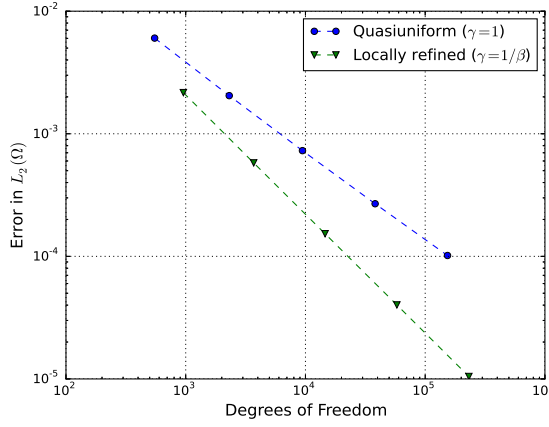}
\end{center}
\end{figure}
%%%%%%%%%%%%%%%%%%%%%%%%%%%%%%%%%%%%%%%%%%%%%%%%%%%%%%%%%%%%%%%%%%%%%
\subsection*{Example 1}
In our first example, we use $\alpha=1/2$ and choose $u_0$~and $f$ 
so that the solution of the initial-boundary value problem 
for~\eqref{eq: fpde} is
\[
u(x,y,t)=\bigl(1+\omega_{\alpha+1}(t)\bigr)
	r^\beta(1-r)\sin(\beta\theta).
\]
In view of \eqref{eq: E(t)}~and \eqref{eq: u(t) mild}, the singular 
behaviour of~$u$ as $r\to0$~or $t\to0$ is typical for such problems.
Figure~\ref{fig: locally-refined} compares the behaviour 
of the $L_2$-error at~$t=1$ for quasi-uniform ($\gamma=1$) and 
locally-refined ($\gamma=1/\beta=3/2$) triangulations.  From
Theorems \ref{thm: non-smooth u0}~and \ref{thm: u0=0}, we expect 
errors of order~$\epsilon(h,1)^2=h^{2\gamma\beta}=h^{4/3}$~and
$\epsilon(h,3/2)^2=h^2\log^2(1+h^{-1})$, respectively.  The number of 
degrees of freedom is of order $h^{-2}$ in both cases, so in 
Figure~\ref{fig: locally-refined} we expect the corresponding error 
curves to be straight lines with gradients $-2/3$~and $-1$, 
which are in fact close to the observed values $-0.7249$~and 
$-0.9707$, respectively, as determined by simple linear least squares 
fits.

%%%%%%%%%%%%%%%%%%%%%%%%%%%%%%%%%%%%%%%%%%%%%%%%%%%%%%%%%%%%%%%%%%%%%
\begin{figure}
\caption{The $L_2$-error as a function of~$t$ for Example~2
with $\alpha=1/2$~and $\gamma=2/\beta$.}
\label{fig: t dependence} 
\begin{center}
\includegraphics[scale=0.4]{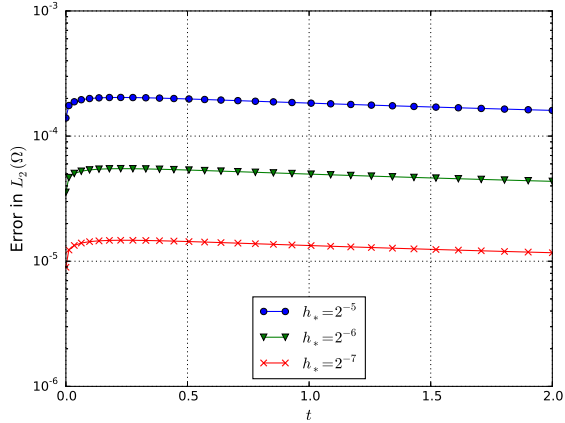}
\end{center}
\end{figure}
%%%%%%%%%%%%%%%%%%%%%%%%%%%%%%%%%%%%%%%%%%%%%%%%%%%%%%%%%%%%%%%%%%%%%
\begin{table}
\caption{$L_2$-Errors and empirical convergence rates (powers
of~$h$) for Example~2 when~$t=1$, with $\gamma=2/\beta$ and 
different choices of~$\alpha$.  (Recall that $N$ denotes the number 
of degrees of freedom in the finite element triangulation.)}
\label{tab: mixed}
\begin{center}
\begin{tabular}{cr|cc|cc|cc}
\multicolumn{2}{c}{}&
\multicolumn{2}{|c|}{$\alpha=1/4$}&
\multicolumn{2}{|c|}{$\alpha=1/2$}&
\multicolumn{2}{|c}{$\alpha=3/4$}\\
\hline
$h_*$   &\multicolumn{1}{c|}{$N$}
               &error    &rate &error    &rate &error    &rate\\
\hline
$2^{-4}$&  1957&1.465e-03&     &1.485e-03&     &1.452e-03&     \\ 
$2^{-5}$&  7593&3.673e-04&1.996&3.723e-04&1.996&3.640e-04&1.997\\ 
$2^{-6}$& 29771&9.471e-05&1.955&9.597e-05&1.956&9.380e-05&1.956\\ 
$2^{-7}$&117039&2.420e-05&1.969&2.451e-05&1.970&2.391e-05&1.972\\ 
$2^{-8}$&466089&6.059e-06&1.998&6.119e-06&2.002&5.931e-06&2.011
\end{tabular}
\end{center}
\end{table}
%%%%%%%%%%%%%%%%%%%%%%%%%%%%%%%%%%%%%%%%%%%%%%%%%%%%%%%%%%%%%%%%%%%%%
\subsection*{Example 2}
In our second example, we impose mixed boundary conditions:  
homogeneous Dirichlet for $\theta=0$~or $r=1$, and homogeneous 
Neumann for~$\theta=\pi/\beta$.  As the initial data we choose the 
first eigenfunction of~$A=-K\nabla^2$,
\[
u_0(x,y)=J_{\beta/2}(\omega r)\sin(\tfrac12\beta\theta),
\]
where $\omega$ is the first positive zero of the 
Bessel function~$J_{\beta/2}$.  We put $f=0$ so (recalling that our 
choice of~$K$ means that the corresponding eigenvalue equals~1) the 
solution is $u(x,y,t)=\E(t)u_0=E_\alpha(-t^\alpha)u_0(x,y)$, and 
choose $\gamma=2/\beta=3$ so that $\epsmix(h,\gamma)$ is of 
order~$h\log(1+h^{-1})$; see \eqref{eq: epsmix}.
Since $A^ru_0\propto u_0\in L_2(\Omega)$ for all~$r>0$, we conclude 
from  Theorem~\ref{thm: u0 smooth} that the 
$L_2$-error~$\|u_h(t)-u(t)\|$ is of order~$h^2\log^2(1+h^{-1})$ 
uniformly for~$0\le t\le T$.  
Figure~\ref{fig: t dependence} confirms this behaviour in the 
case~$\alpha=1/2$.   Finally, Table~\ref{tab: mixed} shows that 
at a fixed positive time~$t=1$ the $L_2$-error does not vary much 
with~$\alpha$.
%%%%%%%%%%%%%%%%%%%%%%%%%%%%%%%%%%%%%%%%%%%%%%%%%%%%%%%%%%%%%%%%%%%%%%
\bibliographystyle{spmpsci}
\bibliography{nonconvexrefs}
%%%%%%%%%%%%%%%%%%%%%%%%%%%%%%%%%%%%%%%%%%%%%%%%%%%%%%%%%%%%%%%%%%%%%
\end{document}